\documentclass[10pt]{article}   
\usepackage{amsthm, amsmath}
\usepackage{amssymb,titlefoot}   
\theoremstyle{plain}

\newtheorem{theorem}{Theorem}[section]
\newtheorem{lemma}[theorem]{Lemma}
\newtheorem{proposition}[theorem]{Proposition}
\newtheorem{corollary}[theorem]{Corollary}

\theoremstyle{definition}

\newtheorem{definition}[theorem]{Definition}
\newtheorem{example}[theorem]{Example}
\newtheorem{algorithm}[theorem]{Algorithm}
\newtheorem{notation}[theorem]{Notation}
\newtheorem{remark}[theorem]{Remark}
\newtheorem{parag}[theorem]{}

\newtheorem{claim}[theorem]{Claim}

\DeclareMathOperator{\Sing}{Sing}
\DeclareMathOperator{\MaxB}{\underline{\mathbf{Max}}}
\DeclareMathOperator{\EMaxB}{E-\underline{\mathbf{Max}}}
\DeclareMathOperator{\ord}{ord}
\DeclareMathOperator{\Eord}{E-ord}
\DeclareMathOperator{\ESing}{E-Sing}
\DeclareMathOperator{\Etop}{E-top}
\DeclareMathOperator{\ECoeff}{E-Coeff}

\DeclareMathOperator{\EReg}{Reg_E}

\title{Desingularization of binomial varieties in arbitrary characteristic. \\ Part II. Combinatorial desingularization algorithm.}
\author{Roc\'{\i}o Blanco\footnote{Partially supported by MMT2007-64704.}}
\date{}

\begin{document}
\maketitle
\pagestyle{headings}

\begin{abstract}

In this paper we construct a combinatorial algorithm of resolution of singularities for binomial ideals, over a field of arbitrary characteristic. This algorithm is applied to any binomial ideal. This means ideals generated by binomial equations without any restriction, including monomials and $p$-th powers, where $p$ is the characteristic of the base field. 

In particular, this algorithm works for toric ideals. However, toric geometry tools are not needed, the algorithm is constructed following the same point of view as Villamayor algorithm of resolution of singularities in characteristic zero.
\end{abstract}

 \\


\section*{Introduction}

The existence of resolution of singularities in arbitrary dimension over a field of characteristic zero was solved by Hironaka in his famous paper \cite{Hironaka1964}. Later on, different constructive proofs have been given, among others, by Villamayor \cite{Villamayor1989}, Bierstone-Milman \cite{BierstoneMilman1997}, Encinas-Villamayor \cite{course}, Encinas-Hauser \cite{strong} and Wodarczyk \cite{Wlodarczyk2005}. 

In positive characteristic, there are some partial results, although the general problem of the existence of resolution of singularities in arbitrary dimension is still open. Recently, the results by Kawanoue \cite{KawanoueI} and his joint work with Matsuki \cite{KawanoueII} begin a sequence of four papers promising resolution of singularities in arbitrary characteristic. In \cite{Villamayor2007} Villamayor gives another approach to the existence of resolution of singularities in positive characteristic, using the previous results included in \cite{Orlando06} about graded algebras as a new tool to attack this problem. See \cite{encinas07} for its application in the case of characteristic zero. 
The most recent work by Bravo-Villamayor \cite{BravoVillamayor2008} and Benito-Villamayor \cite{BenitoVillamayor2008} gives a new procedure to deal with the case of a singular hypersurface embedded in a smooth scheme of positive characteristic. 

In the particular case of binomial ideals, there exist some specific methods of resolution of singularities for binomial varieties with suitable restrictions. In the case of toric ideals (prime binomial ideals), toric geometry tools are often used, such as subdivisions of the associated fan and toric morphisms, to obtain a resolution of singularities. For normal toric varieties over an algebraically closed field of arbitrary characteristic see \cite{toro} and \cite{Cox2000}, and \cite{GonzalezTeissier2002} and \cite{Teissier2004} for non necessarily normal toric varieties.

Bierstone and Milman construct in \cite{lemabm} an algorithm of resolution of singularities, free of characteristic, for reduced binomial ideals with no nilpotent elements. In particular, their algorithm applies to toric ideals. They use Hilbert Samuel function as resolution function, showing that the intersection of the \emph{equimultiple locus} of all the elements of the standard basis of a reduced binomial ideal with no nilpotent elements coincides with its Samuel stratum, see \cite{lemabm} Theorem $7.1$. During this resolution process $p$-th powers are never obtained at the transform ideals. In fact, this algorithm can not treat $p$-th powers of the type $(y^{\gamma}x_1-bx^{\beta})^{p^s}$.  

In this paper we consider binomial ideals without any kind of restriction, and we construct an algorithm of resolution of singularities for these binomial ideals in arbitrary characteristic that provides combinatorial centers of blowing-up. This type of centers preserve the binomial structure of the ideal after blowing-up, what let us ensure the existence of a hypersurface of maximal contact which to make induction on the dimension of the ambient space. The resolution function in which this algorithm is based is given in \cite{part1}.

By blowing up only combinatorial centers we obtain a locally monomial ideal as output. We can apply to this kind of binomial ideal some known resolution algorithm to complete the resolution process. Alternatively, we can apply again the same algorithm. If we apply our algorithm again, we can assure to obtain a log-resolution of the beginning ideal and an embedded desingularization of the corresponding binomial variety with good properties.
\medskip

This paper completes {\it Desingularization of binomial varieties in arbitrary characteristic. \\ Part I. A new resolution function and their properties}, see \cite{part1} for details. The construction of the combinatorial algorithm, that is the aim of this article, needs some technical tools of resolution of singularities. All the technical details related to the construction of the resolution function  in which this algorithm is based are given in \cite{part1}.
\medskip

In section \ref{notdef} we briefly recall some key definitions given in Part I \cite{part1}. The combinatorial algorithm \ref{combalg} is constructed in section \ref{logresbin}. The last section \ref{sec39} is devoted to prove the most important result of this paper, the theorem of embedded desingularization \ref{printeo}.  

\medskip

I thank Santiago Encinas for numerous useful suggestions to improve the presentation of this paper. I am also grateful to Antonio Campillo for his help during all this time.  

\section{Basic definitions} \label{notdef}

We remind here the main definitions and constructions given in \cite{part1}, that are necessary for the construction of the algorithm. See \cite{part1} for details.
\medskip

Let $K$ be an algebraically closed field of arbitrary characteristic, $\textbf{W}$ will be the regular ambient space.  At any stage of the resolution process, $\textbf{W}=\cup_i U_i$, where $U_i\cong\mathbb{A}^n_K$. Locally, inside any affine chart $U_i$, we consider an open set $W$. \\

Let $E=\{V_1,\ldots,V_{r}\}$ be a simple normal crossing divisor in $W$. $E$ defines a stratification of $W$ in the following way: we consider the regular closed sets $$E_{\Lambda}=\bigcap_{\lambda\in \Lambda}V_{\lambda} \text{ where } \Lambda\subseteq \{1,\ldots,r\},$$ 
by definition $E_{\emptyset}:=W$, then each $E_{\Lambda}^{0}=E_{\Lambda}\setminus\left((\cup_{j\not\in \Lambda} V_j)\cap E_{\Lambda}\right)$ is locally closed, regular and $$W=\bigsqcup_{\Lambda}E_{\Lambda}^{0}.$$ Therefore, for every $\xi\in W$ there exists a unique $\Lambda(\xi)\subseteq \{1,\ldots,r\}$ such that $\xi\in E_{\Lambda(\xi)}^{0}$.

\begin{remark} At the beginning of the resolution process $W=Spec(K[x_1,\ldots,x_n])$, $dim(W)=n$. Fix the normal crossing divisor $E=\{V_1,\ldots,V_{n}\}$, where $V_i=V(x_i)$ for each $1\leq i \leq n$, to define a stratification of $W$.
\end{remark}

Let $J\subset K[x]=K[x_1,\ldots,x_n]$ be a binomial ideal (generated by binomial and eventually monomial equations). After a blowing up $W'\rightarrow W$, binomial equations of the type $$1-\mu x^{\delta}, with\  \mu\in K, \delta\in\mathbb{N}^n$$ appear naturally in the transform ideal of $J$. 
The points $\xi'\in W'$ outside the exceptional divisor where $1-\mu x^{\delta}$ vanishes, satisfy $x^{\delta}(\xi')\neq 0$.
 We denote as $y_{i}$ each variable $x_{i}$ that do not vanish anywhere over $V(J)\cap V(1-\mu x^{\delta})$. \\

The binomial equations of $J$ of the form $1-\mu y^{\delta}$ are said to be \emph{hyperbolic equations} of $J$. In what follows we work in localized rings of the type $K[x,y]_y$. 

\begin{remark} \label{cartas} At any stage of the resolution process, inside any chart $U_i$ we consider the open set $$W=Spec(K[x,y]_y)=Spec(K[x_1,\ldots,x_s,y_1,\ldots,y_{n-s}]_y)\subset \mathbb{A}^n_K.$$ 

The normal crossing divisor $E$ is a set of normal crossing regular hypersurfaces in $\mathbb{A}^n_K$, such that  $$E=\{V(x_1),\ldots,V(x_s),V(y_1),\ldots,V(y_{n-s})\}.$$ In the open set $Spec(K[x,y]_y)$ we have $E\cap W=E\cap Spec(K[x,y]_y)=\{V(x_1),\ldots,V(x_s)\}$.
\end{remark}

\begin{definition} \label{jota}
Let $J\subset K[x,y]_y$ be an ideal. We will say $J$ is a \emph{binomial ideal} if it is generated by binomial equations of the type
\begin{equation} J\!=<\!f_1(x,y),\ldots,f_m(x,y)\!> \text{ with } f_i(x,y)\!=\!x^{\lambda_i}(1-\mu_i y^{\delta_i}) \text{ or } f_j(x,y)\!=\!x^{\nu_j}(y^{\gamma_j}x^{\alpha_j}-b_jx^{\beta_j}),  \end{equation} 
with $\alpha_j,\beta_j\in {\mathbb{N}}^n$, $\delta_i,\gamma_j\in {\mathbb{Z}}^n$, $\lambda_i, \nu_j\in {\mathbb{N}}^n$ and $\mu_i,b_j\in K$ for every $1\leq i,j \leq m$.
And where, for each $j$, every equation of the type $y^{\gamma_j}x^{\alpha_j}-b_j x^{\beta_j}$ has no common factors.\\

Denote $|\alpha_j|=\sum_{k=1}^s \alpha_{j,k}$ and $\ |\beta_j|=\sum_{k=1}^s \beta_{j,k}$.
Assume $0<|\alpha_j|\leq|\beta_j|$. 
\end{definition}

\begin{notation} A unique non hyperbolic binomial equation without common factors, will be denoted 
\begin{equation}
f(x,y)=y^{\gamma}x^{\alpha}-bx^{\beta} \text{ where } \alpha,\beta\in\mathbb{N}^n, \gamma\in\mathbb{Z}^n, b\in K \text{ and } 0<|\alpha|\leq |\beta| \label{efegamma} \end{equation} $\alpha=(\alpha_1,\ldots,\alpha_k,0,\ldots,0)$, $\beta=(0,\ldots,0,\beta_{k+1},\ldots,\beta_{k+(s-k)},0,\ldots,0)$ and $\gamma=(0,\ldots,0,\gamma_{1},\ldots,\gamma_{n-s})$ for some $1\leq k\leq s$.  
\end{notation}

\begin{remark} Fixed a normal crossing divisor $E$ as above, we define a modified order function, the $E$-order. Given a binomial ideal $J$, the $E$-order function (associated to $J$), $E\!-\!ord_J$, computes the order of the ideal $J$ along $E\cap W$.
\end{remark}

\begin{definition} \label{Eord}  Let $J\subset \mathcal{O}_{W}$ be a binomial ideal, $W=Spec(K[x,y]_y)$. Let $E=\{V_1,\ldots,V_n\}$ be a normal crossing divisor in $\mathbb{A}^n_K$. Let $\xi\in W$ be a closed point and let $\Lambda(\xi)$ be a subset of $\{1,\ldots,n\}$ such that $\xi\in E_{\Lambda(\xi)}^{0}$. We call \emph{$E$-order} of $J_{\xi}$ in $\mathcal{O}_{W,\xi}$ to the order of the ideal with respect to the $I(E_{\Lambda(\xi)}^{0})_{\xi}$-adic topology $$\Eord_{\mathcal{O}_{W,\xi}}(J_{\xi})=\max\left\{m \in \mathbb{N} /\ J_{\xi}\subset (I(E_{\Lambda(\xi)}^{0})_{\xi})^m \right\}.$$ 
\end{definition}

\begin{definition}
Let $J\subset \mathcal{O}_{W}$ be a binomial ideal as in \ref{jota}. Let $\xi \in W$ be a point. The \emph{$E$-order} function (associated to $J$) is defined as follows, $$\begin{array}{rl} \Eord_J: & W\rightarrow \mathbb{N} \\ & \xi \rightarrow \Eord_J(\xi)=\Eord_{\xi}(J):=\Eord_{\mathcal{O}_{W,\xi}}(J_{\xi}) \end{array}$$ The $E$-order of $J$ at $\xi$ will be denoted $\Eord_{\xi}(J)$. The $E$-order of any binomial equation $f\in J$ at $\xi$, is defined as the $E$-order of the ideal $<f>$ at the point $\xi$. 
\end{definition}

We remind here the technical notion of \emph{binomial basic object along $E$}.

\begin{definition} \label{binbo} An \emph{affine binomial basic object along $E$} (BBOE) is a tuple $B=(W,(J,c),H,E)$ where 
\begin{itemize}
	\item  $W=Spec(K[x,y]_y)=Spec(K[x_1,\ldots,x_s,y_1,\ldots,y_{n-s}]_y)\subset \mathbb{A}^n_K$.
	\item $E$ is a set of normal crossing regular hypersurfaces in $\mathbb{A}^n_K$, such that  $$E=\{V(x_1),\ldots,V(x_s),V(y_1),\ldots,V(y_{n-s})\}.$$ In the open set $Spec(K[x,y]_y)$ we have $E\cap Spec(K[x,y]_y)=\{V(x_1),\ldots,V(x_s)\}$.	
	\item $J$ is a binomial ideal as in (\ref{jota}), and $c$ is a positive integer number.
	\item $H\subset E$ is a set of normal crossing regular hypersurfaces in $W$.
\end{itemize}
\end{definition}  

\begin{definition} A \emph{non affine binomial basic object along $E$} is a tuple $\textbf{B}=(\textbf{W},(\mathcal{J},c),H,E)$ which is covered by affine BBOE. Where 
\begin{itemize}
	\item $\textbf{W}$ is the regular ambient space over a field $K$ of arbitrary characteristic.
	\item $E$ is a set of normal crossing regular hypersurfaces in $\textbf{W}$.
	\item $(\mathcal{J},c)$ is a \emph{binomial pair}, that is, $\mathcal{J}\subset \mathcal{O}_{\textbf{W}}$ is a coherent sheaf of binomial ideals with respect to $E$, as in \ref{jota}, satisfying $\mathcal{J}_{\xi}\neq 0$ for each $\xi\in \textbf{W}$, and $c$ is a positive integer number.
	\item $H\subset E$ is a set of normal crossing regular hypersurfaces in $\textbf{W}$.
\end{itemize}
\end{definition}

The definition of \emph{$E$-singular locus} along a normal crossing divisor $E$ is analogous to the usual definition of singular locus. 

\begin{definition} Let $J\subset \mathcal{O}_{W}$ be an ideal, $c$ a positive integer. We call \emph{$E$-singular locus} of $J$ with respect to $c$ to the set, $$\ESing(J,c)=\{\xi \in W/\ \Eord_{\xi}(J)\geq c\}.$$
\end{definition}

\begin{remark} \label{idpair}
Hironaka introduces the notion of equivalence of \emph{pairs} and using this notion, the definition of \emph{idealistic exponent} or \emph{idealistic pair} as an equivalence class of such pairs. See Hironaka \cite{Hironaka1977} for more details.  
\end{remark}

\begin{remark}\label{equipair}
We always consider pairs $(J,c)$ or binomial basic objects $(W,(J,c),H,E)$ along $E$. This is because of for every point $\xi\in \ESing(J,c)$, the quotient $\frac{\Eord_{\xi}(J)}{c}$ can be defined in terms of the binomial basic object along $E$, modulo the equivalence relation between idealistic exponents. 
\end{remark}
 
\begin{definition} Let $(W,(J,c),H,E)$ be a BBOE, where $W=Spec(K[x_1,\ldots,x_s,y_1,\ldots,y_{n-s}]_y)$ and $E=\{V(x_1),\ldots,V(x_s),V(y_1),\ldots,$ $V(y_{n-s})\}=\{V_1,\ldots,V_n\}$. Let $H=\{H_1,\ldots,H_r\}$ be a normal crossing divisor, $H_i=V(x_j)$ with $1\leq j\leq s$ for each $i$. 

We define a \emph{transformation} of the binomial basic object $$(W,(J,c),H,E)\leftarrow (W',(J',c),H',E')$$ by means of the blowing up $W\stackrel{\pi}{\leftarrow} W'$, along a combinatorial center $Z\subset \ESing(J,c)$, where $W'$ is the strict transform of $W$. With
\begin{itemize}
	\item $H'=\{H_1^{\curlyvee},\ldots,H_r^{\curlyvee},Y'\}$ where $H_i^{\curlyvee}$ is the strict transform of $H_i$ and $Y'$ is the exceptional divisor in $W'$.
	\item $E'=\{V_1^{\curlyvee},\ldots,V_n^{\curlyvee},Y'\}$ where $V_i^{\curlyvee}$ is the strict transform of $V_i$ and $Y'$ is the exceptional divisor in $W'$.
	\item $J'=J^{!}=I(Y')^{\theta-c}\cdot J^{\curlyvee}$ is the \emph{controlled} transform of $J$, where $\theta=\max\ \Eord(J)$ and $J^{\curlyvee}$ is the weak transform of $J$.
\end{itemize}
\end{definition}

\begin{remark} In this context, a combinatorial center is given by the intersection of coordinate hypersurfaces defined by variables $x_i$. 
\end{remark}

\begin{definition} A sequence of transformations of binomial basic objects {\small \begin{equation}   (W^{(0)}\!,(J^{(0)},c),H^{(0)}\!,E^{(0)})\!\leftarrow\!(W^{(1)}\!,(J^{(1)},c),H^{(1)}\!,E^{(1)})\!\leftarrow\!\cdots \leftarrow\!(W^{(N)}\!,(J^{(N)},c),H^{(N)}\!,E^{(N)}) \label{Eresolution} \end{equation}} is a \emph{$E$-resolution} of $(W^{(0)},(J^{(0)},c),H^{(0)},E^{(0)})$ if $\ESing(J^{(N)},c)=\emptyset$. 
\end{definition}

\begin{proposition} \label{usc} Let $J\subset \mathcal{O}_{W}$ be a binomial ideal as in definition \ref{jota}, $\xi\in W$. Then $$\begin{array}{cl} \Eord_J: & W \rightarrow (\mathbb{Z},\leq) \\ & \   \xi \rightarrow \Eord_J(\xi):=\Eord_{\xi}(J)\end{array}$$ is an upper semi-continuous function. 
\end{proposition}

\begin{proof} See \cite{part1}.
\end{proof}

\begin{corollary} Let $J\subset \mathcal{O}_{W}$ be a binomial ideal as in definition \ref{jota}. Let $c$ be a positive integer number. Then $\ESing(J,c)$ is a closed set. 
\end{corollary}

\begin{remark} \label{equivdef}  In addition, the $E$-order function is an \emph{equivariant} function (invariant by the torus action).
\end{remark}

The $E$-order let us to solve the problem of the existence of hypersurfaces of maximal contact in positive characteristic, in the particular case of binomial ideals. We remind here the definition of \emph{hypersurface of E\text{-}maximal contact}.

\begin{definition} Let $J\subset \mathcal{O}_W$ be a binomial ideal as in definition \ref{jota}. Let $\xi\in W$ be a point such that $\Eord_{\xi}(J)=\max\ \Eord(J)=\theta$, $V$ is said to be a hypersurface of  \emph{maximal contact along $E$} for $J$ at $\xi$ (denoted by hypersurface of \emph{E\text{-}maximal contact}) if 
\begin{itemize}
	\item [-] $V$ is a regular hypersurface, $\xi\in V$,
	\item [-] $\ESing(J,\theta) \subseteq V$ and their transforms under blowing up along a center $Z\subset V$ also 
	          satisfy $\ESing(J',\theta) \subseteq V'$, where $J'$ is the controlled transform of $J$ and $V'$ is the strict transform of $V$. 
\end{itemize}
\end{definition}  

\begin{remark} \label{esing} Let $J\subset K[x,y]_y$ be a binomial ideal. Let $\xi \in Spec(K[x,y]_y)$ be a point where $\Eord_{\xi}(J)=\theta>0$ is maximal. It can be proved that in a neighborhood of $\xi$, $$\ESing(J,\theta) \subseteq \{x_i=0\} \text{ for some } 1\leq i\leq s. $$  
\end{remark}

\begin{remark}
As a consequence, the hypersurfaces of $E$-maximal contact will always be given by coordinate equations. The existence of these hypersurfaces it is proved in \cite{part1}. Hence the centers of blowing up will always be combinatorial. 
\end{remark}

We remind briefly the definition and properties of the $E$-resolution function defined in \cite{part1}. To define the $E$-resolution function we rewrite \emph{mobiles language}. See \cite{strong} for more details.
\medskip

Given $(W,(J,c),H,E)$ a BBOE, by induction on the dimension of $W$, construct ideals $J_i$ defined in local flags $W=W_n\supseteq W_{n-1}\supseteq \cdots \supseteq W_i\supseteq  \cdots \supseteq W_1$, and then binomial basic objects $(W_i,(J_i,c_{i+1}),H_i,E_i)$ in dimension $i$, where each $E_i=W_i\cap E$. 

\begin{remark} \label{fac}
If $\ESing(J_i,c_{i+1})\neq\emptyset$ then factorize the ideal $J_i=M_i\cdot I_i$, where each ideal $M_i$ is defined by a normal crossings divisor $D_i$ supported by the current exceptional locus. 
\end{remark}
 
\begin{definition} Let $J_i=M_i\cdot I_i$ be an ideal in $W_i$ at $\xi\in W_i$. Set $\theta_i=\Eord_{\xi}(I_i)$. 
The {\it companion ideal} of $J_i$ at $\xi$, with respect to the critical value $c_{i+1}$ satisfying $\Eord_{\xi}(J_i)\geq c_{i+1}$, is the ideal 
$$P_i= \left\{\begin{array}{ll} I_i & \text{ if }\ \theta_i\geq
c_{i+1} \\ I_i+M_i^{\frac{\theta_i}{c_{i+1}-\theta_i}} & \text{ if }\ 0< \theta_i< c_{i+1} \end{array}\right.$$ 
\end{definition}   

\begin{definition} Let $P_i$ be an ideal in $W_i$. Let $V\subset W_i$ be a hypersurface of $E$-maximal contact for $P_i$ at $\xi \in V$. Denote $\Eord_{\xi}(P_i)=c_i$. The {\it junior ideal} of $P_i$ in $V$ is the ideal
$$J_{i-1}= \left\{
\begin{array}{ll} \ECoeff_V(P_i) & \text{ if } \ECoeff_V(P_i)\neq 0 \\ 1 & \text{ if } \ECoeff_V(P_i)=0 \end{array} \right. $$ where $\ECoeff_{V}(P_i)$ is the coefficient ideal of $P_i$ along $E$ in $V$. See \cite{part1} for details.
\end{definition}

\begin{definition} \label{Einv}
Let $(W,(J,c),H,E)$ be a binomial basic object along $E$. For all point $\xi\in \ESing(J,c)$ the $E$-resolution function $t$ will have $n$ components with lexicographical order, and it will be of one of the following types: $$\begin{array}{ll} (a) &
t(\xi)\!=\!(t_n(\xi),t_{n-1}(\xi),\ldots,t_{n-r}(\xi),\ \infty,\
\infty,\ldots,\infty)
\\ (b) & t(\xi)\!=\!(t_n(\xi),t_{n-1}(\xi),\ldots,t_{n-r}(\xi),
\Gamma(\xi),\infty,\ldots,\infty)
\\ (c) & t(\xi)\!=\!(t_n(\xi),t_{n-1}(\xi),\ldots,t_{n-r}(\xi),\ldots
\ldots \ldots,t_1(\xi)) \end{array}\ \text{ with } t_i(\xi)\!=\!\frac{\theta_i}{c_{i+1}} \text{ if } \theta_i\!>\!0 $$ where $\theta_i=\Eord_{\xi}(I_i)$ and $c_{i+1}=\max\ \Eord(P_{i+1})$ is the critical value in dimension $i$.

In the case $J_i=1$, define $t_i(\xi)=\infty$ and complete the $E$-resolution function $t$ with so many $\infty$ com\-ponents as needed in order to have always the same number of components, that is, $(t_{i-1}(\xi),\ldots,t_1(\xi))=(\infty,\ldots,\infty)$. 

If $\theta_i=0$ then $t_i(\xi)=\Gamma(\xi)$, where $\Gamma$ is the resolution function corresponding to the \emph{monomial case}, see \cite{course}. And complete the $E$-resolution function $(t_{i-1}(\xi),\ldots,t_1(\xi))=(\infty\ldots,\infty)$.
\end{definition}

\begin{lemma} Let $J\subset \mathcal{O}_W$ be a binomial ideal. Let $\xi\in W$ be a point. Let $\mathcal{I}$ be a totally ordered set with the lexicographical order. The function $$\begin{array}{rl} t: & W \rightarrow (\mathcal{I},\leq) \\ & \ \xi \ \rightarrow t(\xi) \end{array}$$ is upper semi-continuous.
\end{lemma}

\begin{corollary} As a consequence, $$\EMaxB(t)=\{\xi\in \ESing(J,c)|\ t(\xi)=\max\ t\}$$ is a closed set. In fact, it is the next center to be blown up. 
\end{corollary}

\begin{remark}
Moreover, by construction $\EMaxB(t)=Z=\cap_{i\in \mathcal{I}}\{x_i=0\}$ with $\mathcal{I}\subseteq \{1,\ldots,n\}$.
\end{remark}

The $E$-resolution function drops lexicographically after blowing up.

\begin{lemma} \label{bajainv} Let $(W,(J,c),H,E)$ be a binomial basic object along $E$, where $J\neq 1$. Let $W \stackrel{\pi}{\leftarrow}W'$ be the blow up along $Z=\EMaxB(t)$. Then $$t(\xi)>t'(\xi')$$ for all $\xi\in Z$, $\xi'\in Y'=\pi^{-1}(Z)$, $\pi(\xi')=\xi$, where $t$ is the $E$-resolution function corresponding to $(W,(J,c),H,E)$ and $t'$ corresponds to $(W',(J',c),H',E')$, its transform by the blow up $\pi$.  
\end{lemma}

\subsection{$E$-resolution of BBOE: Algorithm} \label{algbin} 

The last technical construction that we need to remind is the algorithm of $E$-resolution of binomial basic objects along $E$, given in \cite{part1}.

\begin{theorem} {\bf $E$-resolution of binomial basic objects along $E$.} \label{resobb}
\medskip

An algorithm of $E$-resolution of binomial basic objects of dimension $n$ along a normal crossing divisor $E$ consist of:

\begin{itemize}
	\item[A)] A totally ordered set $(\mathcal{I}_n,\leq)$.
	\item[B)] For each BBOE $(W^{(0)},(J^{(0)},c),H^{(0)},E^{(0)})$ where $dim(W^{(0)})=n$, $J^{(0)}=M^{(0)}\cdot I^{(0)}$ and the ideal $I^{(0)}$ does not contain hyperbolic equations:
\begin{enumerate}
	\item Define an equivariant function $t^{(0)}: \ESing(J^{(0)},c)\rightarrow \mathcal{I}_n$ such that $$\EMaxB t^{(0)}\subset \ESing(J^{(0)},c)$$ is a permissible center for $(W^{(0)},(J^{(0)},c),H^{(0)},E^{(0)})$.  
	\item By induction, assume there exists an equivariant sequence of transformations of BBOE 
\begin{eqnarray}	
(W^{(0)},(J^{(0)},c),H^{(0)},E^{(0)}) \stackrel{\pi_1}{\longleftarrow}\ldots \hspace*{7cm}
\nonumber \\  \ldots\stackrel{\pi_{r-1}}{\longleftarrow} (W^{(r-1)},(J^{(r-1)},c),H^{(r-1)},E^{(r-1)})\stackrel{\pi_r}{\longleftarrow}(W^{(r)},(J^{(r)},c),H^{(r)},E^{(r)})
\label{seq}
\end{eqnarray}
along centers $Z^{(k)}\subset \ESing(J^{(k)},c)$ for $0\leq k\leq r-1$; and equivariant functions $$t^{(k)}: \ESing(J^{(k)},c)\rightarrow \mathcal{I}_n$$ for $0\leq k\leq r-1$, such that $Z^{(k)}=\EMaxB t^{(k)}$. 
	
If $\ESing(J^{(r)},c)\neq \emptyset$ this sequence of transformations can be extended. This means at the $r$-th stage of the $E$-resolution process an equivariant function can be defined $$t^{(r)}: \ESing(J^{(r)},c)\rightarrow \mathcal{I}_n$$ such that $Z^{(r)}=\EMaxB t^{(r)}$ is a permissible center for $(W^{(r)},(J^{(r)},c),H^{(r)},E^{(r)})$.
\end{enumerate}
	\item[C)] For some $r$, the previous sequence of transformations (\ref{seq}) is a $E$-resolution of the original BBOE $(W^{(0)},(J^{(0)},c),H^{(0)},E^{(0)})$, that is, $\ESing(J^{(r)},c)=\emptyset$.
\end{itemize}
\end{theorem}

\begin{remark} Remind that running this algorithm \ref{resobb} of $E$-resolution of a BBOE we only modify the singular points included in the $E$-singular locus. 
\end{remark}

Finally, we recall the main properties of this algorithm \ref{resobb}.

\begin{proposition} \label{propis} 
Fix a BBOE $(W,(J,c),H,E)$ and a $E$-resolution of this BBOE given by theorem \ref{resobb}. This means  $\ESing(J^{(r)},c)=\emptyset$ for some $r\in\mathbb{N}$, $r>0$.
\begin{enumerate}
	\item If $\xi\in \ESing(J^{(k)},c)$ for $0\leq k\leq r-1$, and $\xi\notin Z^{(k)}$ then  $t^{(k)}(\xi)=t^{(k+1)}(\xi')$ where $\pi_{k+1}(\xi')=\xi$. That is, it is possible to identify the points in the $E$-singular loci $$\ESing(J^{(0)},c),\ldots,\ESing(J^{(k)},c)$$ and outside the centers $Z^{(0)},\ldots,Z^{(k)},$ with their corresponding transforms in the $E$-singular locus $\ESing(J^{(k+1)},c)$. 
	\item The $E$-resolution is achieved by means of transformations along centers $\EMaxB t^{(k)}$ for $0\leq k\leq r-1$. The $E$-resolution function $t$ drops after each one of these transformations $$max\ t^{(0)}>max\ t^{(1)}>\ldots >max\ t^{(r-1)}.$$
	\item For all $0\leq k\leq r-1$, the closed set $\EMaxB t^{(k)}$ is equidimensional and regular and its dimension is determined by the value $max\ t^{(k)}$.
\end{enumerate}
\end{proposition}

\section{Log-resolution of binomial ideals} \label{logresbin}

\subsection{Locally monomial resolution of a binomial ideal} \label{sec37} 

\begin{parag} \label{parrafo1}
Let $(W^{(0)},(J^{(0)},c),H^{(0)},E^{(0)})$ be a BBOE. By algorithm \ref{resobb} there exists an index $r$ such that $\ESing(J^{(r)},c)=\emptyset$ where $J^{(r)}=J_n^{(r)}=M_n^{(r)}\cdot I_n^{(r)}$.

If $\Eord(I_n^{(r)})=0$ and $I_n^{(r)}=1$ the resolution process is finished. But if $I_n^{(r)}\neq 1$ then it is necessary to modify the part of the singular locus included in the hyperbolic hypersurfaces which contain $V(I_n^{(r)})$.
\medskip 

Let $\{g_1,\ldots,g_m\}$ be the reduced Gr${\rm \ddot{o}}$bner basis of $J^{(0)}$ in $W^{(0)}$. Let $\xi\in W^{(r)}$ be a point. In a neighborhood of $\xi$, set $I_n^{(r)}=<f_1,\ldots,f_m>$ where $f_j$ comes from the transforms of the generator $g_j$ of the Gr${\rm \ddot{o}}$bner basis of $J^{(0)}$ by the sequence of blow ups. At  $W^{(r)}$ construct the ideal 
$$\tilde{I}_n^{(r)}=(Nhyp(I_n^{(r)}))_y=(<\{f_i\}_{i=1,\ldots,m}/\Eord_{\xi}(f_i)\neq 0\ \forall\ \xi\in W^{(r)}>)_y\subset K[x,y]_y$$

where $Nhyp(I_n^{(r)})$ denotes the ideal generated by the non hyperbolic generators of $I_n^{(r)}$.

Before passing to the localization, it is necessary to rewrite as $y$ the variables $x$ appearing in the hyperbolic generators of $I_n^{(r)}$. 

If $\tilde{I}_n^{(r)}\neq 0$ then by construction $\Eord_{\xi}(\tilde{I}_n^{(r)})>0$ for all $\xi\in W^{(r)}$. Now resolve the binomial pair $(\tilde{I}_n^{(r)},c)$, where $c$ is the corresponding critical value.  
\end{parag}

\begin{remark} The construction of the ideal $\tilde{I}$ depends on the choice of the system of generators of $I$. This is because it is necessary to fix a Gr${\rm \ddot{o}}$bner basis (or any other system of generators) of $J^{(0)}$ from the beginning of the resolution process.
\end{remark}

\begin{example} Let $I=<1-x_1,x_2-x_3^2>\subset K[x_1,x_2,x_3]$ be a binomial ideal, $char(K)=0$. Let $\{1-x_1,x_2-x_3^2\}$ and $\{1-x_1,x_2-x_3^2x_1\}$ be two systems of generators of $I$.

Note that $<x_2-x_3^2>\neq <x_2-x_3^2y_1>$ in $K[y_1,x_2,x_3]_{y_1}$.
\end{example}

\begin{algorithm} {\bf Locally monomial resolution.} \label{combalg}
\medskip
	
Let $J\subset \mathcal{O}_W$ be a binomial ideal without hyperbolic equations, with respect to a normal crossing divisor $E$. Fix a reduced Gr$\ddot{\rm o}$bner basis of $J$. Consider $J=M\cdot I$. At the beginning $\mathcal{O}_W=K[x]$, $E=\{V(x_1),\ldots,V(x_n)\}$ and $J=I$.

Consider the BBOE $(W,(J,c),H,E)$, where $H$ is the set of exceptional hypersurfaces. At the beginning $H=\emptyset$.
\begin{enumerate}
	\item Apply the algorithm \ref{resobb} to $(W,(J,c),H,E)$, where $c=\max \Eord(J)>0$. Obtain $J'=M'\cdot I'$ with $\ESing(J',c)=\emptyset$.  
\item If $\max \Eord(I')=0$
	\begin{itemize}
	\item If $I'=1$ finish. $J'$ principal.
	\item If $I'\neq 1$ take $\tilde{I}$ in $K[x,y]_y$. 
	    \begin{itemize}
	    \item If $\tilde{I}\neq 0$ take $J=\tilde{I}$ and go to step $1$.
	    \item If $\tilde{I}=0$ finish. The ideal $I'$ is given only by hyperbolic equations.
	    \end{itemize}	
\end{itemize}
\item If $\max \Eord(I')>0$ take $J=J'$ and go to step $1$. 
\end{enumerate}
\end{algorithm}

\begin{remark} Step (1) of algorithm \ref{combalg} means modify the singular points in $\ESing(J,c)$. \\
During this step the number of variables $y_i$ does not increase, since the ideal $\tilde{I}$ is cons\-tructed when the $E$-singular locus is empty.
\end{remark}

\begin{remark} \label{localinv}
Let $(W,(J,c),H,E)$ be a binomial basic object along $E$. By construction, the value of the function $t$ at a point $\xi$ of the $E$-singular locus $\ESing(J,c)$ only depends on the point $\xi$. 

Notice that the value of the function $t$ at any point does not depend on the Gr$\ddot{\rm o}$bner basis of the ideal $J$ fixed at the beginning of the $E$-resolution process. This is because the $E$-order of an ideal $I_i$ is independent of the selected set of generators of $I_i$.
\end{remark}

\begin{proposition} \label{Igorro} 
The next combinatorial center to be blown up defined by $\tilde{I}_n^{(r)}$ (\ref{parrafo1}) at $W^{(r)}$ is compatible with the centers defined at other charts in $\textbf{W}$. 
\end{proposition}            

\begin{proof} Without loss of generality we can assume $r$ is the first stage of the $E$-resolution process where it is necessary to define an ideal $\tilde{I}$. Denote by $x_1',\ldots,x_n'$ the variables after blowing up.

At some affine chart $W^{(r)}_i$, let assume the ideal $J^{(r)}=J_n^{(r)}=M^{(r)}\cdot I^{(r)}$ can be written as  $J^{(r)}=<f_{i,1}^{(r)},\ldots,f_{i,s}^{(r)}>$ where each $f_{i,j}^{(r)}=M^{(r)}\cdot g_{i,j}^{(r)}$. By hypothesis, there exists some $g_{i,j}^{(r)}$ which is a hyperbolic equation and $\tilde{I}_n^{(r)}\neq 0$. 

Any other chart with different conditions will be treated before or after $W^{(r)}_i$:
\begin{itemize}
  \item At some affine chart $W^{(l)}_j$, where $\ESing(J^{(l)},c_l)\neq\emptyset$ we consider the center determined by the $E$-resolution function.  	
	\item At some affine chart $W^{(t)}_j$, where $\ESing(J^{(t)},c_t)=\emptyset$ 
\begin{itemize}
	\item and $c_t\!\neq\!c$. If $c_t\!>\!c$ then $W^{(t)}_j$ is the next chart to be considered, otherwise, consider $W^{(r)}_i$.	
	\item and $c_t=c$. If $\Eord(I_n^{(t)})>0$ then at this chart it is still possible to drop the value of $c_t$. If $\Eord(I_n^{(t)})=0$ and $I_n^{(t)}=1$ the resolution process is already finished at this chart. 
	
If $\Eord(I_n^{(t)})=0$, $I_n^{(t)}\neq 1$ and $\tilde{I}_n^{(t)}=0$ the $E$-resolution process is finished at this chart.
\end{itemize}
\end{itemize}

So it is enough to check the assumption in the case of having two such charts. That is, there exists another affine chart $W^{(m)}_k$ such that $\ESing(J^{(m)},c)=\emptyset$, $\Eord(I_n^{(m)})=0$, $I_n^{(m)}\neq 1$ and $\tilde{I}_n^{(m)}\neq 0$. Under these conditions, the ideal $J^{(m)}=M^{(m)}\cdot I^{(m)}$ can be written $J^{(m)}=<f_{k,1}^{(m)},\ldots,f_{k,s}^{(m)}>$ where each $f_{k,j}^{(m)}=M^{(m)}\cdot g_{k,j}^{(m)}$. Some $g_{k,j}^{(m)}$ is a hyperbolic equation and $\tilde{I}_n^{(m)}\neq 0$. 

Suppose $g_{i,1}^{(r)}$ is a hyperbolic equation in $W^{(r)}_i$, then $\max \Eord(g_{i,1}^{(r)})=0$. The following diagram is commutative:
\vspace*{-0.5cm}

$$\begin{array}{cclc} & & W^{(r)}_i=Spec(K[x_1'\ldots,x_n']_{\{x_l'| l\in B_i\}}) & B_i\subset \{1,\ldots,n\} \\ & \swarrow & \hspace*{1cm} \uparrow & \\ W=Spec(K[x_1\ldots,x_n]) & & W^{(r)}_i\cap W^{(m)}_k=Spec(K[x_1'\ldots,x_n']_{\{x_l'| l\in B\}}) &  B\supset B_i\cup B_k \\ & \nwarrow & \hspace*{1cm} \downarrow & \\ & & W^{(m)}_k=Spec(K[x_1'\ldots,x_n']_{\{x_l'| l\in B_k\}}) & B_k\subset \{1,\ldots,n\} \\
\end{array}  \begin{array}{c} g_{i,1}^{(r)} \\ \\ \updownarrow \\ \\ g_{k,1}^{(m)} \end{array}$$

\begin{itemize}
	\item[(1)] If $\max \Eord(g_{k,1}^{(m)})>0$ in $W^{(m)}_k$, look to the points in the intersection $W^{(r)}_i\cap W^{(m)}_k$.
	
At each point of $W^{(r)}_i\cap W^{(m)}_k$, $\Eord(g_{k,1}^{(m)})=0$ since $g_{k,1}^{(m)}$ is the first generator of $I_n^{(m)}$ and its $E$-order is zero in $W^{(r)}_i$. 

  \item[(2)] If $\max \Eord(g_{k,1}^{(m)})=0$ then we construct $\tilde{I}_n^{(r)}$ and $\tilde{I}_n^{(m)}$ at both charts, erasing the first generator $g_{i,1}^{(r)}$ respectively $g_{k,1}^{(m)}$. Continue the argument with the remaining generators.
\end{itemize}
\medskip

If $r$ is not the first stage of the $E$-resolution process where it is necessary to define an ideal $\tilde{I}$, then the ideals $I_n^{(r)}$ and $I_n^{(m)}$ can have different number of generators. 
But in this situation, there was a hyperbolic generator at a previous stage of the $E$-resolution process. Therefore, as above, at each point of the intersection of these charts both generators are erased.
\end{proof}

\begin{claim} The ideal $\tilde{I}_n^{(r)}$ can be non well defined in the intersection of two charts.
\end{claim}

\begin{parag} In such case, to construct $\tilde{I}_n^{(r)}$ and $\tilde{I}_n^{(m)}$, at each chart erase the hyperbolic generators of $I_n^{(r)}$ and $I_n^{(m)}$, respectively.  

In the intersection $U_{i,k}=W^{(r)}_i\cap W^{(m)}_k$ we must consider $\widetilde{\tilde{I}_n^{(r)}\big|_{U_{i,k}}} \text{ and } \widetilde{\tilde{I}_n^{(m)}\big|_{U_{i,k}}}$. That is, erase the hyperbolic generators at each chart, and then consider the ideals restricted to $W^{(r)}_i\cap W^{(m)}_k$. Erase the hyperbolic generators in $\tilde{I}_n^{(r)}\big|_{U_{i,k}}$ and in $\tilde{I}_n^{(m)}\big|_{U_{i,k}}$. Note that in these two steps the same generators are erased in both ideals $I_n^{(r)}$ and $I_n^{(m)}$. Hence $$\widetilde{\tilde{I}_n^{(r)}\big|_{U_{i,k}}}\equiv\widetilde{\tilde{I}_n^{(m)}\big|_{U_{i,k}}}.$$
\end{parag}

\begin{remark} Remind that the \emph{top locus} of an upper semi-continuous function $t$ in $W$ is the reduced closed subscheme of $W$ where $t$ reaches its maximum value, that is, $top(t)=\{\xi\in W\ |\ t(\xi)=\max\ t\}.$  \\  Let $\mathcal{J}$ be a coherent ideal sheaf in $W$, we denote
\begin{itemize}
	\item[-] the set $\Etop(\mathcal{J})=top(\Eord(\mathcal{J}))$ is said to be the \emph{E-top locus} of $\mathcal{J}$.
	\item[-] Let $c$ be a positive integer number, $\Etop(\mathcal{J},c)=\{\xi\in W\ |\ \Eord_{\xi}(\mathcal{J})\geq c\}.$
\end{itemize}
\end{remark}

\begin{remark}
Defining this ideal $\tilde{I_n}$ and considering $J_n=\tilde{I_n}$ in the algorithm \ref{combalg} the following sequence of inclusions is achieved 
$$ Z\subset \cdots \subset \Etop(P_i)\subset \Etop(J_i,c_{i+1})\subset \Etop(P_{i+1})\subset \cdots \subset \Etop(P_n)\subset \Etop(\tilde{I_n},\tilde{c}_{n+1})$$ where $\tilde{c}_{n+1}$ is the suitable critical value, and $P_i$ are the \emph{companion ideals}, see \cite{part1} for details. It holds $\Etop(J_i,c_{i+1})\subset \Etop(P_{i+1})$ since $J_i=\ECoeff(P_{i+1})$.

When $\max \Eord(I_i)=0$ in dimension $i<n$, this chain is not achieved since $\Etop(J_i,c_{i+1})=\Etop(M_i,c_{i+1})$. This is the reason to use $\Gamma$ function. In that case $$Z\subset \Etop(J_i,c_{i+1})\subset \Etop(P_{i+1})\subset \cdots .$$
\end{remark}

\begin{remark} \label{ies} As a consequence $Z\subset \Etop(I_i)$ for all $1\leq i\leq n$, since by construction of the companion ideal $\Etop(P_i)\subseteq \Etop(I_i)$.
\end{remark}

The algorithm \ref{combalg} provides a \emph{locally monomial resolution}, this means the output is a locally monomial ideal. 

\begin{definition} \label{reslocmon}
Let $J\subset\mathcal{O}_W$ be a binomial ideal as in definition \ref{jota}. A \emph{locally monomial resolution} of $J$ is a sequence of blow ups along combinatorial centers $Z^{(k)}$ $$\begin{array}{cccccccc}
(W,H) & \xleftarrow{\ \Pi_{1}\ } & (W^{(1)},H^{(1)}) & \xleftarrow{\ \Pi_{2}\ } & \cdots &
\xleftarrow{\ \Pi_{N}\ } & (W^{(N)},H^{(N)}) \\ J &  & J\mathcal{O}_{W^{(1)}} &  & \cdots & & J\mathcal{O}_{W^{(N)}}
\end{array} \vspace*{-0.3cm}$$ such that 
\begin{itemize}
	\item each center $Z^{(k)}$ has normal crossings with the exceptional divisors $H^{(k)}=\{H_1,\ldots,H_k\}$. In fact, $Z^{(k)}=\cap_{i\in \mathcal{I}}H_i$ where $\mathcal{I}\subseteq \{1,\ldots,k\}$,
	\item the total transform of $J$ at (each affine chart of) $W^{(N)}$ is of the form 
\begin{equation} \label{totaljota} 	J\mathcal{O}_{W^{(N)}}=<\texttt{M}_1\cdot(1-\mu_1y^{\delta_1}),\texttt{M}_2\cdot(1-\mu_2y^{\delta_2}),\ldots,\texttt{M}_r\cdot(1-\mu_ry^{\delta_r}),\epsilon\cdot\texttt{M}_{r+1}> \subset K[x,y]_y
\end{equation} 
with $\texttt{M}_i=I(H_1)^{b_{i_1}}\cdot \ldots \cdot I(H_N)^{b_{i_N}}$ for $1\leq i\leq r+1$ where $H^{(N)}=\{H_1,\ldots,H_N\}$, $b_j\in \mathbb{N}$ for all $1\leq j\leq N$, and $\mu_i\in K$, $\delta_i\in \mathbb{Z}^n$ for $1\leq i\leq r$. And where $\epsilon=1$ if some generator of $J$ is a monomial, and $\epsilon=0$ otherwise. 
\end{itemize}
\end{definition}

\begin{parag} It is necessary to check that the ideal (\ref{totaljota}) is locally a monomial ideal with respect to a regular system of parameters.

Assume $K[x,y]_y=K[x_1,\ldots,x_s,y_1,\ldots,y_{n-s}]_y$. Let $a\in Spec(K[x,y]_y)$ be a point, then $a=(a_1,\ldots,a_n)$ with $a_i\neq 0$ for all $s+1\leq i\leq n$. There is a natural morphism $$K[x,y]_y \rightarrow K[x,y]_{a}=\mathcal{O}_{Spec(K[x,y]_y),a} \rightarrow \widehat{K[x,y]_{a}}$$ where $\widehat{K[x,y]_{a}}$ is the completion of the local ring at the point $a$, denoted by $K[x,y]_{a}$. 
\end{parag}

\begin{remark} When $\tilde{I}$ is constructed, the monomial part is not taken into account. Then the monomials $\texttt{M}_i$ of the total transform (\ref{totaljota}) can contain some variables $x_j$ which have turned into $y_j$ later in the $E$-resolution process, but they have not been rewritten as $y_j$ in these monomials.  
 
In the neighborhood of a point $\xi$ such that $\xi_j=0$, the hyperbolic equations containing $y_j$ disappear, are equal to $1$, so the total transform is written $$J\mathcal{O}_{W^{(N)}}= <\texttt{M}_1\cdot(1-\mu_1y^{\delta_1}),\ldots,\texttt{M}_l\cdot(1-\mu_ly^{\delta_l}),\texttt{M}_{l+1},\ldots,\texttt{M}_r,\epsilon\cdot\texttt{M}_{r+1}>$$ with $l<r$, and $\texttt{M}_i$ can contain the variable $x_j$. After some combinatorial blow ups, $$J\mathcal{O}_{W^{(N_1)}}= <\texttt{M}_1'\cdot(1-\mu_1y^{\delta_1}),\ldots,\texttt{M}_l'\cdot(1-\mu_ly^{\delta_l}),\epsilon\cdot\texttt{M}>.$$

On the other hand, let $a\in Spec(K[x,y]_y)$ be a point where $a_j\neq 0$. Then, in the neighborhood of $a$, $y_j$ is a unit. If the hyperbolic equations vanish at $a$, the variables $y_j$ in the monomial part are units in the local ring $K[x,y]_y$. Thus, we can assume $\texttt{M}_i=x^{\beta_i}$, $\beta_i\in \mathbb{N}^s$ for all $1\leq i\leq r+1$.
\end{remark}

\begin{proposition} \label{escrituralocal} Let $J\subset K[x,y]_y$ be a binomial ideal of the form 
\begin{equation} \label{simpideal}
J=<\texttt{M}_1\cdot(1-\mu_1y^{\delta_1}),\ldots, \texttt{M}_r\cdot(1-\mu_r y^{\delta_r}),\epsilon\cdot\texttt{M}_{r+1}>\subset K[x,y]_y
\end{equation}
where $\texttt{M}_i=x^{\beta_i}$ with $\beta_i\in \mathbb{N}^s$ for $1\leq i\leq r+1$, and $\mu_i\in K$, $\delta_i\in \mathbb{Z}^{n-s}$ for $1\leq i\leq r$, $\epsilon=0,1$. 

If $J\neq 1$, for all $a\in Spec(K[x,y]_y)$ there exist local coordinates $\{z_1,\ldots,z_n\}\in \widehat{K[x,y]_a}$ such that
\begin{equation} \label{zmon}
J\cdot\widehat{K[x,y]_a}=<z^{\lambda_1},\ldots,z^{\lambda_t}>
\end{equation} 
where $z=(z_1,\ldots,z_n), \lambda_i\in\mathbb{N}^n$ for $1\leq i\leq t$, and $t\leq \min(r+1,n)$.
\end{proposition}
 
\begin{proof} By induction on the dimension of the ambient space:
\begin{itemize}
	\item If $n=1$, there is only one variable. Denote it by $x_1$ or $y_1$ depending on the considered point. 
		
In this way, in the neighborhood of a point $\xi\in \mathbb{A}^1_K$ such that $\xi=\xi_1=0$ the ideal $J$ is of the form $J_{\xi}=<\texttt{M}_1,\ldots,\texttt{M}_r,\epsilon\cdot\texttt{M}_{r+1}>$ where $\texttt{M}_i=x_1^{\beta_i}$, $\beta_i\in \mathbb{N}$. Hence $J_{\xi}=<x_1^{\beta}>$ with $\beta=g.c.d(\beta_1,\ldots,\beta_{r+1})$.

If $\xi=\xi_1\neq 0$ then the monomials $\texttt{M}_i=1$. Therefore  $J_{\xi}=<1-\mu_1y_1^{\delta_1},\ldots,1-\mu_r y_1^{\delta_r}>$ with $\mu_i\in K$, $\delta_i\in \mathbb{Z}$, whereas $J_{\xi}\neq 1$ ($\epsilon=0$). 

The ideal $J_{\xi}$ can be rewritten as $J_{\xi}=<1-\eta_1y_1^{\alpha_1},\ldots,1-\eta_r y_1^{\alpha_r}>$ with $\alpha_i\in \mathbb{N}$, $\eta_i\in K$, since all its generators are hyperbolic equations in one variable $y_1$.
	
Set $\alpha=g.c.d(\alpha_1,\ldots,\alpha_r)$. Let check $J_{\xi}=<1-\mu y_1^{\alpha}>$ with $\mu\in K$.

Assume $\alpha_1>\ldots>\alpha_r$. If $\alpha_i=\alpha_j$ for any $i\neq j$ with $\eta_i\neq\eta_j$ then $J_{\xi}=1$ since $$(1-\eta_i y_1^{\alpha_i})-(1-\eta_j y_1^{\alpha_i})=(\eta_j-\eta_i)y_1^{\alpha_i}\in J_{\xi}.$$ 

Since $\alpha_1>\ldots>\alpha_r$, can be easily checked that the ideal $J_{\xi}$ can be expressed 
\begin{equation} \label{Euclidideal}
J_{\xi}=<\!1-\eta_1y_1^{\alpha_1},\ldots,1-\eta_r y_1^{\alpha_r}\!>= <\!1-\frac{\eta_1}{\eta_r}y_1^{\alpha_1-\alpha_r},\ldots,1-\frac{\eta_{r-1}}{\eta_r}y_1^{\alpha_{r-1}-\alpha_r},1-\eta_r y_1^{\alpha_r}>
\end{equation}

Now back to rearrange the exponents and make the same operation as in equation (\ref{Euclidideal}). That is, argue as in the Euclidean algorithm for computing the greatest common divisor, always subtracting the smaller exponent. So 
$$J_{\xi}=<1-\nu_1y_1^{\alpha},\ldots,1-\nu_r y_1^{\alpha}>$$ with $\alpha=g.c.d(\alpha_1,\ldots,\alpha_r)$, $\nu_i\in K$ for $i=1,\ldots,r$.

As we have seen above, it is necessary $\nu_1=\ldots=\nu_r=\mu$ to achieve $J_{\xi}\neq 1$. Then either $J_{\xi}=1$ or $J_{\xi}=<1-\mu y_1^{\alpha}>$. 

If $char(K)=p>0$ and $J_{\xi}\neq 1$, set $\alpha=p^s\alpha'$ with $\alpha'\in\mathbb{N}$ such that $\alpha'\not\equiv 0\ mod\ p$ and $\mu=(\mu')^{p^s}$, $\mu'\in K$, so $1-\mu y_1^{\alpha}=(1-\mu' y_1^{\alpha'})^{p^s}$. Set $z_1=1-\mu' y_1^{\alpha'}$, therefore $J_{\xi}=<z_1^{p^s}>$.

	\item Assume the result holds for a binomial ideal of this form (\ref{simpideal}) in $n-1$ variables. 
	
Let $J_{\xi}$ be a binomial ideal as in (\ref{simpideal}) in $n$ variables. In addition, assume that $\texttt{M}_i$ for all $i=1,\ldots,r+1$, have no common factors. Otherwise $J_{\xi}=M\cdot J_1$ where $J_1$ is of the same form as $J_{\xi}$ without common factors.

When $char(K)=p>0$, let $\delta_i=p^{l_i}\delta_i'$ with $\delta_i'\in \mathbb{Z}^{n-s}$, $l_i\geq 0$, so that for all $1\leq i\leq r$ there exists some $j$, $1\leq j\leq n-s$ such that $\delta_{i_j}'\not\equiv 0\ mod\ p$. Suppose $l_1\leq l_i$ for all $i=2,\ldots,r$, then $$J_{\xi}=<\texttt{M}_1\cdot(1-\mu_1'y^{\delta_1'})^{p^{l_1}},\texttt{M}_2\cdot(1-\mu_2y^{\delta_2}), \ldots,\texttt{M}_r\cdot(1-\mu_ry^{\delta_r}),\epsilon\cdot\texttt{M}_{r+1}>$$ where $(\mu_1')^{p^{l_1}}=\mu_1$. Define $\eta_1=\mu_1'$ and $\alpha_1=\delta_1'=(\alpha_{1_1},\ldots,\alpha_{1_{n-s}})$. Thus  $1-\mu_1'y^{\delta_1'}=1-\eta_1y^{\alpha_1}$ where $\alpha_{1_j}\not\equiv 0\ mod\ p$ for some $1\leq j\leq n-s$. Suppose $j=1$, so $\alpha_{1_1}\not\equiv 0\ mod\ p$. 

Set $z_1=1-\eta_1y^{\alpha_1}$. Formally  $$y_1=\left(\frac{1}{\eta_1}\right)^{\frac{1}{\alpha_{1_1}}}(1-z_1)^{\frac{1}{\alpha_{1_1}}}\cdot y_2^{\frac{-\alpha_{1_2}}{\alpha_{1_1}}} \ldots y_{n-s}^{\frac{-\alpha_{1_{n-s}}}{\alpha_{1_1}}}.$$ Replacing $y_1$ in the other equations {\small $$J_{\xi}\!=<\texttt{M}_1\cdot z_1^{p^{l_1}}, \texttt{M}_2\cdot\left(1-\mu_2\left(\frac{1}{\eta_1}\right)^{\frac{\delta_{2_1}}{\alpha_{1_1}}}(1-z_1)^{\frac{\delta_{2_1}}{\alpha_{1_1}}}\cdot y_2^{\delta_{2_2}-\frac{\alpha_{1_2}\cdot\delta_{2_1}}{\alpha_{1_1}}} \ldots y_{n-s}^{\delta_{2_{n-s}}-\frac{\alpha_{1_{n-s}}\cdot\delta_{2_1}}{\alpha_{1_1}}}\right)\!,\ldots $$} 
{\small $$\ldots,\texttt{M}_r\cdot
\left(1-\mu_r\left(\frac{1}{\eta_1}\right)^{\frac{\delta_{r_1}}{\alpha_{1_1}}}(1-z_1)^{\frac{\delta_{r_1}}{\alpha_{1_1}}}\cdot y_2^{\delta_{r_2}-\frac{\alpha_{1_2}\cdot\delta_{r_1}}{\alpha_{1_1}}} \ldots y_{n-s}^{\delta_{r_{n-s}}-\frac{\alpha_{1_{n-s}}\cdot\delta_{r_1}}{\alpha_{1_1}}}\right),
\epsilon\cdot\texttt{M}_{r+1}>.$$}
	
Note that $y_i,(1-z_1)\in \widehat{K[x,y]_a}$ are units. Since $\alpha_{1_1}\not\equiv 0\ mod\ p$ then $y_i^{\frac{1}{\alpha_{1_1}}},(1-z_1)^{\frac{1}{\alpha_{1_1}}}\in \widehat{K[x,y]_a}$, so that this ideal 
belongs to the following extension of $K[x,y]_y$, inside the completion $$\begin{array}{ccc} K[x,y]_y\rightarrow &  K[x,y_2,\ldots,y_{n-s},y_2^{\frac{1}{\alpha_{1_1}}},\ldots,y_{n-s}^{\frac{1}{\alpha_{1_1}}}]_{y,y^{\frac{1}{\alpha_{1_1}}}}[[z_1]] & \rightarrow \widehat{K[x,y]_a} \\ & \shortparallel \vspace*{-0.2cm} & \\ & K[x,y_2^{\frac{1}{\alpha_{1_1}}},\ldots,y_{n-s}^{\frac{1}{\alpha_{1_1}}}]_{y^{\frac{1}{\alpha_{1_1}}}}[[z_1]] & \end{array}$$  where $y^{\frac{1}{\alpha_{1_1}}}=\{y_2^{\frac{1}{\alpha_{1_1}}},\ldots,y_{n-s}^{\frac{1}{\alpha_{1_1}}}\}$. 

Passing to the quotient $K[x,y_2^{\frac{1}{\alpha_{1_1}}},\ldots, y_{n-s}^{\frac{1}{\alpha_{1_1}}}]_{y^{\frac{1}{\alpha_{1_1}}}}[[z_1]]/<z_1>$ 
$$\hspace*{-0.5cm} \overline{J_{\xi}}= < \texttt{M}_2\cdot\left(1-\mu_2\left(\frac{1}{\eta_1}\right)^{\frac{\delta_{2_1}}{\alpha_{1_1}}} y_2^{\delta_{2_2}-\frac{\alpha_{1_2}\cdot\delta_{2_1}}{\alpha_{1_1}}} \ldots y_{n-s}^{\delta_{2_{n-s}}-\frac{\alpha_{1_{n-s}}\cdot\delta_{2_1}}{\alpha_{1_1}}}\right),\ldots $$ 
$$ \hspace*{0.5cm} \ldots,\texttt{M}_r\cdot
\left(1-\mu_r\left(\frac{1}{\eta_1}\right)^{\frac{\delta_{r_1}}{\alpha_{1_1}}} y_2^{\delta_{r_2}-\frac{\alpha_{1_2}\cdot\delta_{r_1}}{\alpha_{1_1}}} \ldots y_{n-s}^{\delta_{r_{n-s}}-\frac{\alpha_{1_{n-s}}\cdot\delta_{r_1}}{\alpha_{1_1}}}\right),
\epsilon\cdot\texttt{M}_{r+1}>.$$
So that $\overline{J_{\xi}}\subset K[x,y_2^{\frac{1}{\alpha_{1_1}}},\ldots, y_{n-s}^{\frac{1}{\alpha_{1_1}}}]_{y^{\frac{1}{\alpha_{1_1}}}}$. By induction hypothesis $\overline{J_{\xi}}$ is of the desired form.	 
\end{itemize} \vspace*{-0.8cm}
\end{proof} 

\begin{remark} The ideal $J_a=<z^{\lambda_1},\ldots,z^{\lambda_t}>\subset \widehat{K[x,y]_a}$ can be written in this form in a $\grave{\rm e}$tale neighborhood of the point $a$, denoted by $U_{z,a}$. Note that $$Spec(\widehat{K[x,y]_a})\subset U_{z,a}\subset Spec(K[x,y]_y)$$ and $\mathcal{O}_{U_{z,a}}$ is a finite extension of $K[x,y]_y$.

In the intersection of two such $\grave{\rm e}$tale neighborhoods $U_{z,a_1}\cap U_{z,a_2}$, it holds $J_{a_1}=J_{a_2}$. 
\end{remark}

\begin{proposition} Let $J\subset K[x,y]_y$ be a binomial ideal as in (\ref{simpideal}). The local writing given by proposition \ref{escrituralocal} is invariant by the torus action. 
\end{proposition}
\begin{proof}
In the neighborhood of a point $\xi\in U_{z,\xi}\subset Spec(K[x,y]_y)$, if the ideal $J\cdot\widehat{K[x,y]_{\xi}}$ is of the form $$J\cdot\widehat{K[x,y]_{\xi}}=<z^{\lambda_1},\ldots,z^{\lambda_t}>$$ then in a neighborhood of the point $\mathcal{T}^d(\xi)\in Spec(K[x,y]_y)$ there exist local coordinates $\{w_1,\ldots,w_n\}$ in $\widehat{K[x,y]_{\mathcal{T}^d(\xi)}}$ such that $J\cdot\widehat{K[x,y]_{\mathcal{T}^d(\xi)}}=<w^{\lambda_1},\ldots,w^{\lambda_t}>$. 

It is enough to note that if $\xi=(\xi_1,\ldots,\xi_n)$ then $\mathcal{T}^d(\xi)=(t^{a_1}\xi_1,\ldots,t^{a_n}\xi_n)$, with $a_i\in \mathbb{Z}^d$, $i=1,\ldots,n$.
\end{proof}

\begin{corollary} The locally monomial resolution given by algorithm \ref{combalg} is invariant by the torus action.  
\end{corollary}

\begin{remark} Note that the monomials $z^{\lambda_i}$ generating the ideal of equation (\ref{zmon}) in proposition \ref{escrituralocal} are supported on the union of the exceptional divisors (coming from the $E$-resolution process) and the irreducible components of the original binomial variety given by the ideal $J$.  
\end{remark}

\begin{remark} At the beginning we consider binomial varieties which can not be described globally by a monomial ideal.
\end{remark}

\subsection{Log-resolution} \label{sec38}

Given a binomial ideal $J$, the algorithm \ref{combalg} provides a locally monomial resolution of $J$. Our aim is to achieve a log-resolution of $J$. The step from the locally monomial resolution to the log-resolution modifies the singular points included in the hyperbolic hypersurfaces. 

\begin{remark} \label{opciones}
To transform a locally monomial ideal into an exceptional monomial ideal it can be applied: 
\begin{itemize}
	\item[A)] The algorithm by Goward \cite{Goward2005}.
	\item[B)] Villamayor algorithm of resolution of singularities (over fields of characteristic zero), adapted to the case of an ideal generated by monomials. In this case the algorithm works over a field of arbitrary characteristic.
	\item[C)] The algorithm by Bierstone and Milman constructed in \cite{lemabm} to resolve ideals generated by monomials.
	\item[D)] The algorithm \ref{resobb} of $E$-resolution of BBOE. In this case, the resolution achieved is invariant by the torus action.  
\end{itemize}
\end{remark}

\begin{parag}
Note that algorithm \ref{resobb} of $E$-resolution of BBOE can be applied to a BBOE whose ideal is generated by monomials. This is because in the neighborhood of a point $\xi\in Spec(K[x,y]_y)$ there exists a $\grave{\rm e}$tale neighborhood $U_{z,\xi}$ of $\xi$ such that the extension $K[x,y]_y\subset\mathcal{O}_{U_{z,\xi}}$ is finite. 

Then, apply algorithm \ref{resobb} (in the $\grave{\rm e}$tale neighborhood $U_{z,\xi}$) to the ideal $J_{\xi}\subset \mathcal{O}_{U_{z,\xi}}$ which is a monomial ideal in $\{z_1,\ldots,z_n\}$. The centers to be blown up are combinatorial in $z$, and in terms of variables $x,y$, they are intersections of hyperbolic hypersurfaces and coordinate hypersurfaces. 
\medskip

In this case, the normal crossing divisor to be fixed, denoted by $E^{*}$, is $E^{*}=\{V(z_1),\ldots,V(z_n)\}$. Recall that for ideals given by monomials, the $E$-order function is the usual order function. 
\end{parag}

\begin{corollary} {\bf Log-resolution}
\medskip

Let $J\subset\mathcal{O}_W$ be a binomial ideal as in \ref{jota}, without hyperbolic equations, respect to a normal crossing divisor $E$. Algorithm \ref{combalg} provides a locally monomial resolution of $J$. 

Option $D$ of remark \ref{opciones} gives a \emph{log-resolution} of $J$, that is, a sequence of blow ups at regular centers $$\hspace*{-0.8cm} (W,H,E)\leftarrow (W^{(1)},H^{(1)},E^{(1)})\leftarrow \cdots \leftarrow (W^{(r)},H^{(r)},E^{(r)})\approx (U_z,\emptyset,E^{*})\leftarrow $$ $$ \leftarrow  (W^{(r+1)},H^{(r+1)},E^{(r+1)})=(U_z^{(1)},H^{(r+1)},(E^{*})^{(1)})\leftarrow \cdots \leftarrow (W^{(N)},H^{(N)},E^{(N)})$$ such that 
\begin{itemize}
	\item each center $Z^{(k)}$ has normal crossings with the exceptional divisors $H^{(k)}$
	\item the total transform of $J$ in $W^{(N)}$ is of the form $$J\mathcal{O}_{W^{(N)}}=I(H_1)^{b_1}\cdot \ldots \cdot I(H_m)^{b_m}$$ with $m\in \mathbb{N}$, $b_i\in \mathbb{N}$ for all $1\leq i \leq m$ and $H^{(N)}=\{H_1,\ldots,H_m\}$. 
\end{itemize}
\end{corollary}

\begin{proof} Algorithm \ref{combalg} provides a locally monomial resolution of $J$. By proposition  
\ref{escrituralocal} rewrite the resulting ideal. So that locally, in a $\grave{\rm e}$tale neighborhood, algorithm \ref{resobb} can be applied.

This means it is possible to apply algorithm \ref{combalg} again. But there is a substantial difference, if after some blow ups we achieve $J'=M'\cdot I'$ then $$max\ \Eord(I')=max\ \ord(I')=0 \Leftrightarrow I'=1$$ hence this gives a log-resolution of $J$. 
\end{proof}

\section{Embedded desingularization} \label{sec39}

\begin{lemma} \label{regE} 
Let $X$ be a closed subscheme of $W=Spec(K[x_1,\ldots,x_s,y_1,\ldots,y_{n-s}]_y)\subset\mathbb{A}^n_K$. The ideal $I(X)$ is a binomial ideal as in \ref{jota}, without monomial generators, respect to a normal crossing divisor $E=\{V(x_1),\ldots,V(x_s),V(y_1),\ldots,V(y_{n-s})\}=E_x\sqcup E_y$ where $E_x=\{V(x_1),\ldots,V(x_s)\}$, $E_y=\{V(y_1),\ldots,V(y_{n-s})\}$. 

Let $\EReg(X)=\{\xi\in X|\ X\text{ is regular at }\xi\text{ and has normal crossings with }E \}$ be the \emph{regular locus of $X$ along $E$}, then $$\EReg(X)\cap E_x=\emptyset.$$ Note that $X\cap E_y=\emptyset$ since $X\subset W=Spec(K[x,y]_y)$ and $E\cap Spec(K[x,y]_y)=E_x$.
\end{lemma}

\begin{proof} 

If $\xi\in X\subset Spec(K[x,y]_y)$ then $y(\xi)\neq 0$. Thus $\xi\not\in L$ for all $L\in E_y$, moreover if $\xi\in V\in E$ then $\xi\in E_x$. 

Let $\xi\in X$ be a point such that $X$ is regular at $\xi$ and $\xi\in V=V(x_i)\in E_x$. Then $\xi_i=0$ and there exists a generator $f(x,y)=y^{\gamma}x^{\alpha}-bx^{\beta}\in I(X)$ as in (\ref{efegamma}) such that $f(\xi)=0$ and $\alpha_i>0$ or $\beta_i>0$. 

If $\alpha_i>0$ it holds $f(x,y)=y^{\gamma}x_i^{\alpha_i}x^{\alpha^{*}}-bx^{\beta}$ with $\alpha^{*}=\alpha-(0,\ldots,0,\alpha_i,0,\ldots,0)\in\mathbb{N}^s$. By the Jacobian criterion $X$ has no normal crossings with $V(x_i)$ at $\xi$. Therefore $X$ has no normal crossings with $E$ at $\xi$. Analogously for $\beta_i>0$. 

If there is not any generator of $I(X)$ under these conditions then the variable $x_i$ can be eliminated. Consider the same problem in dimension $n-1$.
\end{proof}

\begin{corollary} Let $X$ be a closed subscheme of $W=Spec(K[x,y]_y)$ where $I(X)$ is a binomial ideal as in \ref{jota}, without monomial generators, respect to a normal crossing divisor $E$. 

If $\xi\in \EReg(X)$ then in a neighborhood of $\xi$ $$I(X)_{\xi}=<1-\mu_1y^{\delta_1},\ldots,1-\mu_ry^{\delta_r}>$$ where $\mu_i\in K$, $\delta_i\in \mathbb{Z}^{n-s}$, for some $r\geq 1$.
\end{corollary}

\begin{proof} By lemma \ref{regE} it holds $I(X)_{\xi}\subset K[y]_y$. The variables $x$ do not vanish at $\xi$, then $I(X)_{\xi}$ is a binomial ideal in terms of the variables $y$.
\end{proof} 

As a consequence, the following property of algorithm \ref{combalg} holds. 

\begin{proposition} \label{propi}
Let $(W,(J,c),H,E)$ be a BBOE such that $J$ is the ideal of a regular subvariety $X$ along $E$ of pure dimension, $H=\emptyset$ and $c=1$, then the resolution function $t$ is constant. 

Note that $t$ is defined here by means of option $D$ of remark \ref{opciones}.
\end{proposition}

\begin{proof}
Assume $dim(X)=n$ and $J$ is a binomial ideal as in \ref{jota} generated only by binomials. If $\xi\in X$ then $X$ has normal crossings with $E$ at $\xi$. Hence, by lemma \ref{regE} $\xi\not\in V$ for all hypersurface $V\in E_x$, in order to avoid tangency. Then, in a neighborhood of $\xi$, $$I(X)_{\xi}=<1-\mu_1y^{\delta_1},\ldots,1-\mu_ry^{\delta_r}>$$ where $\mu_i\in K$, $\delta_i\in \mathbb{Z}^{n-s}$, for some $r\geq 1$, and the $E$-order of this ideal is zero at all points of the neighborhood. 

Applying algorithm \ref{combalg}, in the first step $\widetilde{I(X)_{\xi}}=0$ since $I(X)_{\xi}$ is already a locally monomial ideal. 
	
The subvariety $X$ is regular at the point $\xi$, so at least one of these hyperbolic equations is not a $p$-th power of the characteristic $p=char(K)$. Argue as in the proof of \ref{escrituralocal} we can write $$I(X)_{\xi}=<z_1,1-\mu_2^{*}(1-z_1)^{\alpha_2}y^{\delta_2^{*}},\ldots,1-\mu_r^{*}(1-z_1)^{\alpha_r}y^{\delta_r^{*}}>$$ where $\mu_i^{*}\in K$, $\alpha_i\in \mathbb{Q}$, $\delta_i^{*}\in \mathbb{Q}^{n-s}$ with $\delta_{i_1}^{*}=0$ for all $i$. 

By induction $I(X)_{\xi}=<z_1,\ldots,z_l>$ for some $l\leq n$ where $\{z_1,\ldots,z_n\}$ are local coordinates in the completion of the local ring at the point $\xi$. 

Then, along $\{V(z_1),\ldots,V(z_n)\}$ it holds $$t(\xi)=(\overbrace{1,\ldots,1}^{l},\infty,\ldots,\infty).$$
This argument works for every regular point $\xi\in X$. Since all the components of $X$ have the same dimension, there are always $l$ coordinates equal to $1$ at the resolution function. Therefore $t$ is constant.
\end{proof}

\begin{theorem} {\bf Embedded desingularization} \label{printeo}

Let $X\subset W=Spec(K[x,y]_y)$ be a closed reduced subscheme where $I(X)$ is a binomial ideal as in \ref{jota}, generated only by binomials with respect to a normal crossing divisor $E=E_x\sqcup E_y$. 

Then there exists a sequence of transformations of pairs $$(W,H=\emptyset)\leftarrow (W^{(1)},H^{(1)})\leftarrow \cdots \leftarrow (W^{(N)},H^{(N)})$$ which induces a proper birational morphism $\Pi: W^{(N)} \rightarrow W$ such that 
\begin{enumerate}
	\item The morphism $\Pi$ restricted to the regular locus of $X$ along $E$, defines an isomorphism $$\EReg(X)\cong \Pi^{-1}(\EReg(X))\subset W^{(N)}.$$  
	\item $X^{(N)}$, the strict transform of $X$ in $W^{(N)}$, is regular and has normal crossings with $H^{(N)}$. 
	\item \emph{Equivariance}: If there is a torus action on $(X\subset W,H)$ then there is also a torus action on $(X^{(N)}\subset W^{(N)},H^{(N)})$.
\end{enumerate}
\end{theorem}

\begin{proof} The proof of the analogous result in characteristic zero can be found in \cite{EncinasVillamayor2003}. 

 Let $(W^{(0)},(J^{(0)},1),H^{(0)},E^{(0)})$ be a binomial basic object along $E^{(0)}$, where $W^{(0)}=W$, $J^{(0)}=I(X)$, $H^{(0)}=\emptyset$ and $E^{(0)}=E$. Note that $$\ESing(J^{(0)},1)=\{\xi\in W^{(0)}|\ \Eord_{\xi}(J^{(0)})\geq 1\}=X\cap E_x\subset X=Sing(J^{(0)},1).$$ 

By lemma \ref{regE}, $\EReg(X)\cap E_x=\emptyset$ then $$\ESing(J^{(0)},1)\cap \EReg(X)=\emptyset$$ Note $\EReg(X)\neq\emptyset$ since $X$ is reduced.
\medskip

Therefore when we apply the algorithm \ref{resobb} to the BBOE $(W^{(0)},(J^{(0)},1),H^{(0)},E^{(0)})$ and we blow up along $Z^{(0)}\subset \ESing(J^{(0)},1)$ the points at $\EReg(X)$ are never modified. 
\medskip

Hence, setting $c=1$ and applying algorithm \ref{combalg}, we achieve a locally monomial resolution of $(W^{(0)},(J^{(0)},1),H^{(0)},E^{(0)})$. This means a sequence of blow ups along combinatorial centers $Z^{(k)}$  $$\begin{array}{ccccccc}(W^{(0)},(J^{(0)},1),H^{(0)},E^{(0)}) & \hspace*{-0.3cm} \stackrel{\pi_1}{\leftarrow} & \hspace*{-0.3cm} (W^{(1)},(J^{(1)},1),H^{(1)},E^{(1)}) & \hspace*{-0.3cm} \stackrel{\pi_2}{\leftarrow} & \hspace*{-0.3cm} \ldots & \hspace*{-0.3cm} \stackrel{\pi_r}{\leftarrow} & \hspace*{-0.3cm} (W^{(r)},(J^{(r)},1),H^{(r)},E^{(r)}) \\ J^{(0)} & \hspace*{-0.3cm} & \hspace*{-0.3cm} J^{(0)}\mathcal{O}_{W^{(1)}} & \hspace*{-0.3cm} & \hspace*{-0.3cm} \cdots & \hspace*{-0.3cm} & \hspace*{-0.3cm} J^{(0)}\mathcal{O}_{W^{(r)}} \\ \EReg(X) & \hspace*{-0.3cm} & \hspace*{-0.3cm} \pi_1^{-1}(\EReg(X)) & \hspace*{-0.3cm} & \hspace*{-0.3cm} \cdots & \hspace*{-0.3cm} & \hspace*{-0.3cm} \pi_r^{-1}(\ldots(\pi_1^{-1}(\EReg(X)))\ldots) \end{array}$$ such that 
\begin{itemize}
	\item each center $Z^{(k)}\subset \ESing(J^{(k)},1)$ has normal crossings with the exceptional divisors $H^{(k)}=\{H_1,\ldots,H_k\}$. In fact $Z^{(k)}=\cap_{i\in \mathcal{I}}H_i$ where $\mathcal{I}\subseteq \{1,\ldots,k\}$,
	\item the total transform of $J^{(0)}$ at (each affine chart of) $W^{(r)}$ is of the form $$J^{(0)}\mathcal{O}_{W^{(r)}}=<\texttt{M}_1\cdot(1-\mu_1y^{\delta_1}),\texttt{M}_2\cdot(1-\mu_2y^{\delta_2}),\ldots,\texttt{M}_r\cdot(1-\mu_r y^{\delta_r}),\epsilon\cdot\texttt{M}_{r+1}>$$ as in (\ref{totaljota}).
\end{itemize}
In addition, by the above argument  $$\emptyset\neq\EReg(X)\cong\pi_1^{-1}(\EReg(X))\cong\ldots\cong\pi_r^{-1}(\ldots(\pi_1^{-1}(\EReg(X)))\ldots)$$ that is, $\EReg(X)\cong\pi^{-1}(\EReg(X))$ where $\pi=\pi_r\circ\ldots\circ\pi_1$.
\medskip

Once the locally monomial resolution is achieved, by proposition \ref{escrituralocal}, for all $\xi\in W^{(r)}$ there exist local coordinates $\{z_1,\ldots,z_n\}\in \widehat{K[x,y]_{\xi}}$ such that 
$$J^{(0)}\mathcal{O}_{W^{(r)}}\cdot\widehat{K[x,y]_{\xi}}=<z^{\lambda_1},\ldots,z^{\lambda_t}>$$ where $z=(z_1,\ldots,z_n)$, $\lambda_i\in\mathbb{N}^n$ for $1\leq i\leq t$, and $t\leq \min(r+1,n)$. 
\medskip

Consider the normal crossing divisor $E^{*}=\{V(z_1),\ldots,V(z_n)\}$ in $W^{(r)}$. Factorize the total transform $J^{(0)}\mathcal{O}_{W^{(r)}}\cdot\widehat{K[x,y]_{\xi}}=M\cdot Q$ where $M$ is a monomial ideal with support at $E^{*}$.  

Then, in a $\grave{\rm e}$tale neighborhood of the point $\xi$, $U_{z,\xi}$, consider the BBOE $$(U_{z,\xi}\subset W^{(r)},(Q,1),\emptyset,E^{*}).$$ Note that $\ESing(Q,1)=\Sing(Q,1)$ since $Q$ is a monomial ideal with respect to $\{z_1,\ldots,z_n\}$. 

The isomorphism $\EReg(X)\cong\pi^{-1}(\EReg(X))$ provides $$\pi^{-1}(\EReg(X))\subset \Sing(Q,1)=\ESing(Q,1)$$ since the points at $\EReg(X)$ can not be included in the support of $M$. 

Let $X^{(r)}$ be the strict transform of $X$ in $U_{z,\xi}\subset W^{(r)}$. Then $$\pi^{-1}(\EReg(X))\subset X^{(r)}$$ and it is a non empty open dense subset of $X^{(r)}$. 

Apply the algorithm \ref{resobb} of $E$-resolution of BBOE to $(U_{z,\xi},(Q,1),\emptyset,E^{*})$. The output is the following sequence of transformations $$\hspace*{-1cm} (U_{z,\xi},(Q,1),\emptyset,E^{*}) \stackrel{\pi_{r+1}}{\leftarrow}\ldots\stackrel{\pi_N}{\leftarrow}  (W^{(N)},(J^{(N)},1),H^{(N)},E^{(N)})\stackrel{\pi_{N+1}}{\leftarrow}\ldots $$ $$\ldots \stackrel{\pi_{N_1}}{\leftarrow} (W^{(N_1)},(J^{(N_1)},1),H^{(N_1)},E^{(N_1)})$$

at permissible centers $Z^{(k)}=\EMaxB t^{(k)}=\MaxB t^{(k)}$ and such that for some index $N_1$ this sequence is a resolution, that is, $\Sing(J^{(N_1)},1)=\emptyset$. 

Observe that $\MaxB t^{(k)}$ denotes the set $\MaxB t^{(k)}=\{\xi\in\Sing(J^{(k)},1)|\ t^{(k)}(\xi)=\max t^{(k)}\}$.
\medskip

By property \ref{propi}, the resolution function at some stage $j$, with $j<N_1$, $t^{(j)}:\Sing(J^{(j)},1)\rightarrow \mathcal{I}_l$ is constant along $\EReg(X)$, and takes the value $$(\overbrace{1,\ldots,1}^{l},\infty,\ldots,\infty)$$ where $l$ is the codimension of $\EReg(X)$. 

Since the resolution function drops after blowing up, there exists a unique index $N$ such that $max\ t^{(N)}=(1,\ldots,1,\infty,\ldots,\infty)$ and the maximal value of the resolution function for $k<N$ is $max\ t^{(k)}>(1,\ldots,1,\infty,\ldots,\infty)$.  

Hence $\EReg(X)\cong\Pi^{-1}(\EReg(X))$ where $\Pi=\pi_N\circ\ldots\circ\pi_r\circ\ldots\circ\pi_1$ because of, up to now, we only have modified points where the maximal value of the resolution function was strictly bigger than $(1,\ldots,1,\infty,\ldots,\infty)$.
\medskip  

Therefore $\Pi^{-1}(\EReg(X))\subset Z^{(N)}=\MaxB t^{(N)}$ and moreover the strict transform of $X$, $X^{(N)}\subset Z^{(N)}$ then $$\Pi^{-1}(\EReg(X))\subset X^{(N)}\subset Z^{(N)}$$ is an open dense in $X^{(N)}$ having the same codimension as $X^{(N)}$. Hence $X^{(N)}=\bigcup_i C_i$ is a union of connected components $C_i$ of $Z^{(N)}$ such that $$C_i\subset \overline{\Pi^{-1}(\EReg(X))}=X^{(N)}$$ is regular and has normal crossings with $E^{(N)}$. As a consequence, $C_i$ has normal crossings with $H^{(N)}\subset E^{(N)}$, and therefore $X^{(N)}$ is regular and has normal crossings with $H^{(N)}$. \\

Since algorithm \ref{combalg} is equivariant, this embedded desingularization is also equivariant. 
\end{proof}

\begin{remark}
Note that the processes of resolution of singularities of various (local)
charts of affine BBOE patch up to form a unique process of resolution of singularities of the
non affine BBOE. 
\end{remark}

\begin{remark} The different processes of resolution of singularities
of charts patch up since the resolution function is a local invariant (remark \ref{localinv}) and every center of blowing up is compatible with the centers defined at other charts (proposition \ref{Igorro}). 
\end{remark}

\bibliographystyle{plain}

\providecommand{\bysame}{\leavevmode\hbox to3em{\hrulefill}\thinspace}

\ \\
Roc\'{\i}o Blanco \\
Universidad de Castilla-La Mancha. Departamento de Matem\'aticas. \\ E.U. de Magisterio.
  Edificio Fray Luis de Le\'on.\\ Avda. de los Alfares 42, 16071 Cuenca, Spain.\\
mariarocio.blanco@uclm.es
\noindent 

\end{document}